\documentclass[a4paper]{article}
\usepackage[utf8x]{inputenc}
\usepackage{amsthm}
\usepackage{float}
\usepackage{mathtools,nccmath}
\usepackage{amssymb}
\usepackage{amsmath}
\usepackage{amsfonts}
\usepackage{latexsym}
\usepackage{enumitem}
\usepackage{url}
\usepackage{bbold}
\usepackage{xcolor}
\usepackage{tkz-tab}
\usepackage[linesnumbered,ruled,vlined]{algorithm2e}
\usepackage{algorithmic}
\usepackage{tikz}
\usetikzlibrary{calc,fit,spy}
\usepackage{pgfplots}
\usetikzlibrary{arrows.meta}
\usepgfplotslibrary{fillbetween}
\pgfplotsset{width=7cm,
        compat=1.5.1,
        standard/.style={
        axis x line=middle,
        axis y line=middle,
        enlarge x limits=0.15,
        enlarge y limits=0.15,
        every axis x label/.style={at={(current axis.right of origin)},anchor=north west},
        every axis y label/.style={at={(current axis.above origin)},anchor=north east}
    }
}
\usepackage{multirow}

 \newcommand\ForAuthors[1]%          %  temporary remark for the
 {\par\smallskip                     %  authors:
  \begin{center}%                    %
   \fbox%                            %    --------
   {\parbox{0.9\linewidth}%          %    |  #1  |
    {\raggedright\sc--- #1}%         %    --------
   }%                                %
  \end{center}%                      %
  \par\smallskip                     %
 }
  \newcommand\comment[1]{}

\title{A Parametric Optimization Point-Of-View of Comparison Functions}

\author{Assalé Adjé\\ LAboratoire de Modélisation Pluridisciplinaire et Simulations\\
LAMPS\\ Université de Perpignan Via Domitia\\ France
}

\def\rr{\mathbb R}
\def\br{\overline{\rr}}
\def\rd{\rr^d}
\def\rm{\rr^m}
\def\nn{\mathbb N}
\def\norm#1{\| #1\|}

\def\agmx{\operatorname{Argmax}}
\def\agmn{\operatorname{Argmin}}

\def\supfg#1#2{{#1}^{\vee #2}}
\def\inffg#1#2{{#1}^{\wedge #2}}

\def\kclsi{\mathcal{K}_\infty}

\newtheorem{proposition}{Proposition}

\newtheorem{theorem}{Theorem}
\newtheorem{corollary}{Corollary}

\newtheorem{example}{Example}

\begin{document}
\maketitle

\begin{abstract}
In this paper, we obtain results about the positive definiteness, the continuity and the level-boundedness of two optimal value functions of specific parametric optimization problems. Those two optimization problems are generalizations of special functions called comparison functions which arise in dynamical systems stability theory. 
\end{abstract}

%\begin{keyword}

%\end{frontmatter}

\section{Introduction}

The paper exposes a problem arising in stability theory of dynamical systems (see e.g.~\cite{bhatia2006dynamical}) and in  parametric optimization (see e.g.~\cite{bank1982non}). The problem studied here is the analysis of some continuous properties of a function defined as the optimal value of an objective function over the sublevel of a fixed real-valued function for which the level is the function variable. On the other hand, when the constraint function is a norm, this parametric function is used as a \textit{comparison function}(see~\cite{DBLP:journals/mcss/Kellett14} for a comprehensive survey). In this context, when the objective function is continuous, strictly positive except at zero (where it is zero) and level-bounded, the constructed parametric function belongs to $\kclsi$. The set $\kclsi$ is the set of continuous functions on strictly positive reals which are strictly positive except at zero where it is zero, strictly increasing and which tend to plus infinity when the norm tends to infinity (see the discussions in~\cite[p. 98]{hahn1967stability} and~\cite[A.5]{khalil2002nonlinear}). Such functions permit to link norms of trajectories of a dynamical system with Lyapunov functions which constitute functional certificates to the stability of the dynamical system. Comparison functions are involved into the construction of converse Lyapunov theorems~\cite{classicalConverseKellet}. In general, the role of those theorems is to obtain the equivalence between Lyapunov functions existence and certain notions of stability.

The contribution of the paper in the parametric optimization field is to propose continuity results for the case where there is only one functional constraint and the level associated is a parameter. The results proposed in the paper do not require the nonemptiness of the strict sublevel set and do not involve convexity assumptions on the objective or the constraint function contrarily to~\cite{b073ecea-b76c-3262-ac64-b852e5e8abd0,5677d28d-1966-3228-887c-28440390ed89}.
Interested reader can consult~\cite{10.1007/BF02055196,bank1982non} for presentations of continuity results for optimal value functions of parametric optimization problems.  From comparison functions point-of-view, the purpose is to generalize the construction of comparison functions proposed by Hahn and Khalil for discontinuous functions. The continuity is useful for converse Lyapunov theorems and stability notions. The key argument to use comparison functions is their bijectivity as they are strictly increasing and continuous. For optimization problems on the trajectories of discrete-time systems~\cite{adje2018quadratic,adje2021quadratic,ahmadi2024robust}, we can extend the problem to discontinuous objective functions. We could exploit comparison functions to compute lower or upper  bounds for the objective functions. Therefore, the generalization of discontinuous comparison type functions is a preliminary work of the analysis of these new optimization problems. 

At first sight, the approach is standard and some of the results are based on Berge's maximum theorem~\cite[p.116-117]{berge1963topological} and some reformulations~\cite{thibault2023unilateral}. One of the result proposes a direct approach for which the set-valued mappings approach does not help.

The paper is organized as follows. Section~\ref{background} is devoted to necessary definitions and results on optimization theory and set-valued mappings theory. Section~\ref{results} presents the main results of the paper. We will find continuity results of the generalization of comparison type functions. We will also propose a study of the positive definiteness of those functions. Section~\ref{conclusion} concludes the paper and opens new potential applications.  

\section{Preliminaries}
\label{background}
We quickly recall necessary standard results for optimization theory and the necessary background to introduce Berge maximum theorem.

\subsection{Key Tools for Optimization}
\subsubsection{Single-Valued Functions Semicontinuity}
Let $X$ be a nonempty subset of $\rd$. We recall that a function $f:X\mapsto \br:=\rr\cup \{-\infty\}\cup \{+\infty\}$ is said to be {\it lower} (resp. {\it upper}) {\it semicontinuous} at $x\in X$ if either $f(x)=-\infty$ (resp. $f(x)=+\infty$) or for all $\varepsilon>0$ there exists $\eta>0$ satisfying for all $y\in X$ such that $\norm{x-y}\leq \eta$, we have $f(y)\geq f(x)-\varepsilon$ (resp. $f(y)\leq f(x)+\varepsilon$). From the definitions, we see immediately that a function $f$ is upper semicontinuous  at $x$ if and only if $-f$ is lower semicontinuous at $x$. We can define equivalently semicontinuity in terms of limits: $f$ is lower (resp. upper) semicontinuous at $x\in X$ if and only if for all $(x_k)_{k\in\nn}\subset X$ which converge to $x$ we have $\liminf_{k\to +\infty} f(x_k)\geq f(x)$ (resp. $\limsup_{k\to +\infty} f(x_k)\leq f(x))$.

Finally, the function is said to be {\it lower} (resp. {\it upper}) {\it semicontinuous} on $X$ if it is lower (resp. upper) semicontinuous at all $x\in X$. 

Let $\alpha$ be in a real,  we denote by $[f\leq \alpha]_X$ the {\it sublevel} of $f$ at level $\alpha$ i.e. the set $\{x\in X : f(x)\leq \alpha\}$. Similarly, we denote by $[\alpha\leq f]_X$ the {\it superlevel} at level $\alpha$ of $f$ that is the set $\{x\in X : \alpha\leq f(x)\}$. Global semicontinuity can be defined using sublevels or superlevels : $f$ is lower (resp. upper) semicontinuous on $X$ if and only if for all $\alpha \in\rr$, $[f\leq \alpha]_X$ (resp. $[\alpha\leq f]_X$) are closed. By extension, we will use the notation, for all $\alpha,\beta\in\rr$ such that $\alpha\leq \beta$, $[\alpha\leq f\leq \beta]_X$ for the set $\{x\in X : \alpha\leq f(x)\leq \beta\}$. We will remove the subscript $X$ when $X=\rd$.

Finally, $f:X\mapsto \rr$ is continuous at $x\in X$ if and only if it is lower and upper semicontinuous at $x\in X$.

Semicontinuity plays an important role in optimization theory. Indeed, the following optimization problems:
\[
\inf_{x\in C} f(x) \qquad \text{ and } \qquad \sup_{x\in C} g(x)
\]
both admit optimal solutions when $f$ is lower semicontinuous, $g$ is upper semicontinuous and $C$ is nonempty and compact. 
Sometimes we will use the notations $\inf_C f$ (resp. $\sup_C f$) to the infimum of $f$ over the set $C$ (resp. the supremum of $f$ over the set $C$).
\subsubsection{Level-Boundedness}
Let $X$ be an unbounded subset of $\rd$. A function $f:X\mapsto \rr$ is said to be level-bounded if and only if $\lim_{\norm{x}\to +\infty} f(x)=+\infty$. Hence, level-boundedness can be written as for all real $M$, there exists a positive real $r$ such that $\norm{x}>r$ implies that $g(x)> M$.
Finally, this is equivalent to the fact that for all $\alpha\in\rr$, the sets $[g\leq \alpha]$ are bounded (possibly empty).

From the optimization point-of-view, the level-boundedness can replace the compactness. Indeed, any lower semicontinuous level-bounded function $f:\rd\mapsto \rr$ is lower bounded on $\rd$ and has a minimizer. Furthermore, if we minimize a lower semicontinuous function over a nonempty sublevel of a level-bounded function, then the minimization problem has a finite optimal value and an optimal solution on the closure of the sublevel. 

Finally, obviously a level-bounded function $f$ satisfies $\sup_{x\in\rd} f(x)=+\infty$ and if $f$ is also continuous, then the equation $f(x)=s$ has a solution for all $s\geq \inf_{x\in\rd} f(x)$.    
\subsection{Set-valued Maps}
Our framework is placed in real finite dimensional vector spaces but the notions presented here can be generalized to topological spaces.
Let $X$ and $Y$ be two nonempty subsets of topological spaces. A set-valued map $M$ from $X$ to $Y$ (or a multivalued map or a multifunction) is denoted by $M:X \rightrightarrows Y$ and associates to $x\in X$ a subset of $Y$. For example, for any function $g:\rd\mapsto \rr$, the map which associates to $s\in\rr$ the set 
$\{x\in \rd : g(x)\leq s\}=[g\leq s]$ is a set-valued map. The domain of the set-valued map $M:X \rightrightarrows Y$ is the set $\{x\in X : M(x)\neq \emptyset\}$ whereas its graph is the set $\{(x,y)\in X\times Y : y\in M(x)\}$.

To introduce the Berge's maximum theorem, we need to define some continuity notions on set-valued maps. A set-valued map $M$ is said to be {\it lower hemicontinuous} at $x\in X$ if and only if for all open sets $U$ such that $M(x)\cap U\neq \emptyset$, there exists a neighborhood $V$ of $x$ (in $X$) such that $M(y)\cap U\neq \emptyset$ for all $y \in V$. The set-valued map is lower hemicontinuous on $X$ if it is lower hemicontinuous at each $x\in X$. A set-valued map $M$ is said to be {\it upper hemicontinuous} at $x\in X$ if and only if for all open sets $U$ such that $M(x)\subseteq U$, there exist a neighborhood $V$ of $x$ (in $X$) such that $M(y)\subseteq U$ for all $y \in V$. The set-valued map is upper hemicontinuous on $X$ if it is upper hemicontinuous at each $x\in X$. As for functions, a set-valued map which is lower and upper hemicontinuous at $x$ is said to be continuous at $x$. Again as for semicontinuity of functions, we have equivalences to deal with global hemicontinuity notions.

\begin{proposition}[Lower Hemicontinuity]
\label{lower}
Let $M$ be a set-valued map from $X$ to $Y$ (both are subsets of some topological spaces). The following assertions are equivalent:
\begin{enumerate}
\item $M$ is lower hemicontinuous;
\item For all open sets $U$, the set $\{x\in X : M(x)\cap U\neq \emptyset\}$ is open in $X$;
\item For all closed sets $F$, the set $\{x\in X : M(x)\subseteq F\}$ is closed in $X$; 
\end{enumerate}
\end{proposition}

\begin{proposition}[Upper Hemicontinuity]
\label{upper}
Let $M$ be a set-valued map from $X$ to $Y$  (both are subsets of some topological spaces). The following assertions are equivalent:
\begin{enumerate}
\item $M$ is upper hemicontinuous;
\item For all open sets $U$, the set $\{x\in X : M(x)\subseteq U\}$ is open in $X$;
\item For all closed sets $F$, the set $\{x\in X : M(x)\cap F\neq \emptyset\}$ is closed in $X$; 
\end{enumerate}
\end{proposition}

In Propositions~\ref{lower} and~\ref{upper}, the two last equivalences come from the fact that $\{x\in X : M(x)\cap C^c\neq \emptyset\}^c=\{x\in X : M(x)\subseteq C\}$.

In~\cite{b073ecea-b76c-3262-ac64-b852e5e8abd0}, the authors define lower and upper semicontinuous at each point where the image of the set-valued mapping is a compact set. In this paper, we use the general version of those notions developed in~\cite[Chap. I]{thibault2023unilateral} (the interested reader can also consult~\cite[Chap. 6]{beer1993topologies} or~\cite[Chap. 16]{infinitedim}). 

As we are placed in finite dimensional spaces, closedness are completely determined by sequences. Using the closedness characterizations of hemicontinuity, we will need to connect $M(x_n)$ and $M(x)$ when $(x_n)_{n\in\nn}$ converges to $x$. One way to get this link is to use Painlevé-Kuratowski (denoted PK) limits of set sequences. We can reason with $\liminf$ and $\limsup$. Let us consider a sequence of sets $(C_n)_{n\in\nn}$ in $X\subseteq \rm$. The PK-$\liminf$ of $C_n$ is the set of the limits of convergent sequences $(x_n)$ such that $x_n\in C_n$ for all $n\in\nn$ whereas the PK-$\limsup$ of $C_n$ is the set of the limits of convergent subsequences $(x_\psi(n))$ for which $x_{\psi(n)}\in C_{\psi(n)}$ for all $n\in\nn$. The sequence $(C_n)_{n\in\nn}$ is said to have a PK-$\lim$ if its PK-$\liminf$ coincide with its PK-$\limsup$. In particular, an increasing (in the sense of the inclusion) sequence of sets has a PK-$\lim$ which consists in the closure of the union of sets. The PK-$\liminf$ and PK-$\limsup$ are used to define other continuity notions for set-valued mappings called inner and outer semicontinuity~\cite[Chap. I]{thibault2023unilateral}.

Now, we are ready to present the Berge's maximum theorem.

\begin{theorem}[Berge maximum theorem]
Let $T$ and $X$ be topological spaces. Let $f:T\times X\mapsto \br$ be an extended real-valued functuion and $M:T\rightrightarrows X$ a set-valued map. We define the marginal function $\Phi:T\mapsto \br$ for all $t\in T$ by:
\[
\Phi(t):=\sup_{x\in M(t)} f(t,x)
\]
Then:
\begin{enumerate}
\item If the set-valued map $M$ is lower semicontinuous at $\overline{t} \in T$ and if the function $f$ is lower semicontinuous on $\{\overline{t}\} \times M(\overline{t})$, then the function $\Phi$ is lower semicontinuous at $\overline{t}$.
\item If the set-valued map $M$ is upper semicontinuous at $\overline{t}$, if the function $f$ is upper semicontinuous at every point of $\{\overline{t}\} \times M(\overline{t})$, and if the set $M(\overline{t})$ is compact, then the function $\Phi$ is upper semicontinuous at $\overline{t}$.
\end{enumerate}
\end{theorem}

\section{Positive Definiteness, Continuity and Level-Boundedness of Comparison Type Functions}
\label{results}
Hahn, in~\cite[p.98]{hahn1967stability}, proposed to construct, for a positive definite\footnote{(positive) for all $x\in\rd$, $f(x)\geq 0$ and (definite)$f(x)=0$ if and only if $x=0$}, level-bounded (the term {\it radially bounded} is more familiar in control/stability theory and this term was used by Hahn) and continuous function $f$, two functions defined for all $s\geq 0$:  
\begin{equation}
\label{hahndef}
\underline{\alpha}^f(s):=\inf_{s\leq \norm{y}} f(y) \text{ and } \overline{\alpha}^f(s):=\sup_{\norm{y}\leq s} f(y)
\end{equation}
to obtain bounds over $f$:
\[
\underline{\alpha}^f(\norm{x})\leq f(x)\leq \overline{\alpha}^f(\norm{x})
\]
Hahn and Khalil~\cite[A.5]{khalil2002nonlinear} propose unformal proofs about the continuity of $\underline{\alpha}^f$ and $\overline{\alpha}^f$. Simpler properties to prove are inherited from the properties of $f$. The functions $\underline{\alpha}^f$ and $\overline{\alpha}^f$ are positive definite and level-bounded. The fact that those function are increasing is also a straightforward property.

The goal of the paper is to generalize the construction replacing the norm by a general function $g:\rd\mapsto\rr$. Then, we study the monotonicity and the positive definiteness of the obtained functions. Finally, we study the continuity and the level-boundedness properties of the obtained functions with respect to the relative properties of $f$ and $g$.

\subsection{Definitions and Useful Facts}

Let $f:\rd\mapsto \rr$ and $g:\rd\mapsto \rr$. We define the function $\supfg{f}{g}:\rr\mapsto \br$ as follows:
\begin{equation}
\label{supfgdef}
\supfg{f}{g}:s\mapsto \sup_ {g(x)\leq s} f(x)
\end{equation}
Note that we have those particular situations:
\begin{itemize}
\item If $s<\inf_{\rd} g$ then $\supfg{f}{g}(s)=-\infty$ as the supremum over the empty set is equal to $-\infty$.
\item If $s=\inf_{\rd} g$ then $\supfg{f}{g}(s)=\sup_{\agmn(g)} f$ and this optimal value is again equal to $-\infty$  if and if only $\agmn(g)$ is empty. 
\item The supremum in the definition of $\supfg{f}{g}$ is taken over a nonempty set when $s>\inf_{\rd} g$ and the supremum is actually taken over $\rd$ when $s\geq \sup_{\rd} g$. 
\end{itemize} 
Dually, we can define the function $\inffg{f}{g}:\rr\mapsto \br$ as follows:
\begin{equation}
\label{inffgdef}
\inffg{f}{g}:s\mapsto \inf_ {s\leq g(x)} f(x)
\end{equation}
Note that, some results obtained for $\supfg{f}{g}$ can be easily extended to $\inffg{f}{g}$ using the relation:
\begin{equation}
\label{infsuprel}
\inffg{f}{g}=\inf_{s\leq g(x)} f(x)=-\sup_{-g(x)\leq -s} -f(x)=-(\supfg{(-f)}{(-g)}\circ (x\mapsto -x))
\end{equation}   
Similarly to $\supfg{f}{g}$, $\inffg{f}{g}$ is known in particular situations. If $s$ is greater than $\inf_{\rd} g$, then $\inffg{f}{g}(s)=\inf_{\rd} f$. If $s$ is strictly greater than $\sup_{\rd} g$, $\inffg{f}{g}(s)$ is taken over a nonempty set. If $s$ is equal to $\sup_{\rd} g$, $\inffg{f}{g}(s)$ is equal to $\inf_{\agmx(g)} f$ and thus is equal to $+\infty$ if and only if $\agmx(g)$ is empty. Finally, $\inffg{f}{g}(s)$ is equal to $+\infty$ for all $s$ strictly greater than $\sup_{\rd} g$.
\subsection{Monotonicity and Positive Definiteness}
We present two algebraic properties : the increasing property and the positive definiteness property.
\subsubsection{Monotonicity}
The monotonicity of $\supfg{f}{g}$ comes from the fact that $\rd \supset C:\mapsto \sup_C f$ is an increasing function of $C$ with respect to the set inclusion order. The monotonicity of $\inffg{f}{g}$ can be deduced either from the fact that $\rd \supset C:\mapsto \inf_C f$ is an decreasing function of $C$ with respect to the set inclusion order or from the
composition relation presented at Eq.~\eqref{infsuprel}. We present those results into Prop.~\ref{monotony}.

\begin{proposition}[Monotonicity]
\label{monotony}
The functions $\inffg{f}{g}$ and $\supfg{f}{g}$ are increasing.
\end{proposition}
\subsubsection{Positive Definiteness}
First, we can see that $\inffg{f}{g}$ and $\supfg{f}{g}$ cannot be positive definite. This is implied by the fact that those functions are increasing. If  $\inffg{f}{g}(0)=0$ (resp. $\supfg{f}{g}(0)=0$) then for all $s<0$, $\inffg{f}{g}(s)\leq 0$ (resp. $\supfg{f}{g}(s)\leq 0$). Due to Prop.~\ref{monotony}, we can only expect that :
\begin{equation}
\label{infdef}
\inffg{f}{g}(0)=0 \text{ and } \inffg{f}{g}(s)>0 \text{ for all } s>0 
\end{equation}
\begin{equation}
\label{supdef}
\supfg{f}{g}(0)=0 \text{ and } \supfg{f}{g}(s)> 0 \text{ for all }s>0
\end{equation}
Conditions~\eqref{infdef} and~\eqref{supdef} say respectivelly that $\inffg{f}{g}$ and $\supfg{f}{g}$ are positive definite on the nonnegative reals.

For the function $\supfg{f}{g}$, to obtain Condition~\eqref{supdef} can be easily reformulated algebraically.
\begin{proposition}[Positive definiteness of $\supfg{f}{g}$ ]
\label{positivedefsup}
Condition~\eqref{supdef} holds if and only if the three following conditions hold:
\begin{enumerate}
\item For all strictly positive real $s$, $[f>0]\cap [g\leq s]\neq \emptyset$;
\item The set $[g\leq 0]$ is nonempty and included in $[f\leq 0]$;
\item There exists a sequence $(x_n)_{n\in\nn}\in [g\leq 0]$ such that $\lim_{n\to +\infty} f(x_n)=0$. 
\end{enumerate}
\end{proposition}
\begin{proof}
This is just an equivalent reformulation of Condition~\eqref{supdef}.
\end{proof}
When $f$ is positive definite, to check Condition~\eqref{supdef} becomes even easier.
\begin{proposition}[Conditions when $f$ positive definite] 
If $f$ is positive definite on $\rd$, Condition~\eqref{supdef} is equivalent to for all $s>0$, the set $[g\leq s]$ is not reduced to $\{0\}$ and $[g\leq 0]=\{0\}$.
\end{proposition}
\begin{proof}
If $f$ is positive definite then $[f>0]=\rd\backslash\{0\}$ and then the assertion $[f >0]\cap [g\leq s]\neq \emptyset$ for all $s>0$ becomes $[g\leq s]$ contains a nonzero vector. The set $[f\leq 0]=\{0\}$ and then $\emptyset \neq [g\leq 0]\subseteq [f\leq 0]$ becomes $[g\leq 0]=\{0\}$. Finally, the null sequence satisfies the third statement of Prop.~\ref{positivedefsup}.
\end{proof}
\begin{corollary}
If $f$ is positive definite then $\overline{\alpha}^f$ is positive definite.
\end{corollary}
We can reformulate Condition~\eqref{infdef} easily similarly to Prop~\eqref{positivedefsup}.
\begin{proposition}[Positive definite of $\inffg{f}{g}$ - 1]
\label{positivedefinf}
Condition~\eqref{infdef} holds if and only if the three following statement hold:
\begin{enumerate}
\item For all $s>0$ there exists $b_s>0$ such that $[g\geq s]\subseteq [f\geq b_s]$;
\item The set $[g\geq 0]$ is nonempty and included in $[f\geq 0]$;
\item There exists a sequence $(x_n)_{n\in\nn}\in [g\geq 0]$ such that $\lim_{n\to +\infty} f(x_n)=0$.
\end{enumerate}
\end{proposition}
\begin{example}
\label{exmupper}
Let us define $f:\rr\mapsto \rr$ as follows : $f(x)=1$ if $x\leq 0$ and $f(x)=x^2$ for all $x>0$. We suppose that $g$ is the identity on $\rr$. 
Then :
\[
\supfg{f}{g}(s)=\left\{
\begin{array}{lr} 
1 &  \text{ if } s\leq 1\\
s^2 & \text{ if } s>1\\
\end{array}
\right.
\quad \text{ and }\quad 
\inffg{f}{g}(s)=\left\{
\begin{array}{lr} 
0 &  \text{ if } s\leq 0\\
s^2 & \text{ if } s>0\\
\end{array}
\right.
\] 
The function $f$ is strictly positive and then $[f\leq 0]$ is empty and the second statement of Prop~\ref{positivedefsup} does not hold and thus $\supfg{f}{g}$ cannot be positive definite on $\rr_+$. However, $[g\geq 0]$ is nonempty and $[f\geq 0]=\rd$. Hence the second statement of Prop~\ref{positivedefinf} holds. The third statement holds by continuity at right at 0 and $f(0^+)=0$. The first statement holds as for all $s>0$, $f$ is lower bounded by $s^2$.          
\qed        
\end{example}

Contrary to $\supfg{f}{g}$, the fact that $f$ is positive definite does not bring anything except that we can remove the second statement of Prop.~\ref{positivedefinf}. The two other assertions cannot be reformulated. To obtain simpler results to check whether $\inffg{f}{g}$ satisfies Condition~\eqref{infdef}, we add semicontinuity assumptions on $f$ and $g$.

\begin{proposition}[Positive definiteness of $\inffg{f}{g}$ - 2]
\label{positivedef}
If the following conditions hold:
\begin{enumerate}
\item The function $f$ is lower semicontinuous and nonnegative;
\item The function $g$ is upper semicontinuous with $\inf_{\rd} g\leq 0$;
\item For all $s\geq 0$, $\displaystyle{\inffg{f}{g}(s)=\inf_{[s\leq g]\cap C_s} f}$ for some nonempty compact set $C_s$;
\item $\emptyset\neq\{x\in\rd : f(x)=0\}\subseteq \{x\in\rd : g(x)=0\}$.
\end{enumerate}
Then $\inffg{f}{g}$ satisfies Condition~\eqref{infdef}.
\end{proposition}
\begin{proof}
As $f$ is lower semicontinuous and $[g\geq s]\cap C_s$ is nonempty and compact there exists $x\in[g\geq s]\cap C_s$
such that $f(x_s)=\inffg{f}{g}(s)$. As $f$ is nonnegative $f(x_s)\geq 0$. If $f(x_s)=0$ then we have $g(x_s)=0$ which concludes the proof.
\end{proof}
\begin{proposition}
Suppose that $f$ is lower semicontinuous, positive definite and level-bounded  and  $g$ is upper semicontinuous and positive definite. Then Condition~\eqref{infdef} holds.
\end{proposition}

\begin{proof}
The first statement of Prop.~\ref{positivedef} holds by assumption on $f$. As $g$ is positive definite $g(0)=0$ and $g(x)$ is strictly positive for all $x\neq 0$. Then $\inf_{\rd} g=g(0)=0$ and the second statement of Prop.~\ref{positivedef} holds. Let $s\geq 0$, then $[g\geq s]$ is nonempty and let us pick $y \in [s\leq g]$. We have $\inffg{f}{g}(s)=\inf_{[g\geq s]\cap [f\leq f(y)]} f$ with ${[g\geq s]\cap [f\leq f(y)]}\neq \emptyset$. The set $[f\leq f(y)]$ is compact as $f$ is lower semicontinuous and level-bounded which proves that the third statement of Prop.~\ref{positivedef} holds. Finally, as $f$ and $g$ are definite then $\{x\in\rd : f(x)=0\}=\{0\}=\{x\in\rd : g(x)=0\}$ which validates the last statement of Prop.~\ref{positivedef} and ends the proof.
\end{proof}

\begin{corollary}
For all lower semicontinuous positive definite function $f$, the function $\underline{\alpha}^f$ satisfies Condition~\eqref{infdef}. 
\end{corollary}

\subsection{Continuity and Level-Boundeness}

We, now, present the continuity results for $\supfg{f}{g}$ and $\inffg{f}{g}$. We also study the asymptotic behaviour of them when $s$ goes to $+\infty$.   
\begin{theorem}[Continuity results for $\supfg{f}{g}$]
\label{continuity}
Let $f:\rd\mapsto \rr$ and $g:\rd\mapsto \rr$. Then:
\begin{enumerate}
\item If $f$ and $g$ are lower semicontinuous, $\supfg{f}{g}$ is lower semicontinuous;
\item If $g$ is lower-semicontinuous and level-bounded and $f$ is upper semicontinuous, then $\supfg{f}{g}$ is upper semicontinuous, finite-valued and $\sup$ in the definition of $\supfg{f}{g}$ can be replaced by a $\max$. If, moreover, $f$ is level-bounded then $\lim_{s\to +\infty} \supfg{f}{g}(s)=+\infty$.
\end{enumerate}
Finally, if $f$ is continuous and $g$ is lower semicontinuous and level-bounded then $\supfg{f}{g}$ is continuous.
\end{theorem}

\begin{proof}[Proof of Th.~\ref{continuity}]
{\itshape 1}. Let $x\in\rd$. Note that if for all $y\in\rr^m$, $F(x,y)=f(x)$, then $\{(x,y)\in\rd\times \rr^m : F(x,y)\leq t\}=\{x\in\rd: f(x)\leq t\}\times \rr^m$. Thus, $F$ is lower/upper semicontinuous if and only if $f$ is. Now, to use Berge's theorem, we must prove that the multivalued map $M(s):=\{x\in\rd : g(x)\leq s\}$ for all $s\in\rr$ is lower hemicontinuous. Let $C$ be a closed subset of $\rd$, let $s\in\rr$ and $(s_n)_{n\in\nn}$ such that $(s_n)_{n\in\nn}$ converges to $s$ and $M(s_n)\subseteq C$ for all $n\in\nn$. Let us prove that $M(s)\subseteq C$. We can suppose that $s_n < s$ for all $n\in\nn$; otherwise $M(s)\subseteq M(s_n)$ for some $n\in\nn$ and the result holds. So we can extract an increasing subsequence $(s_{n_k})_{k\in\nn}$ converging to $s$. Then $M(s_{n_k})$ PK converges to $M(s)$ which is closed as $g$ is supposed to be lower semicontinuous and finally $M(s_n)$ converges to $M(s)$. We conclude, as $M(s)$ is the union of $M(s_n)$, that $M(s)\subseteq F$ and $M$ is lower hemicontinuous and by Berge's theorem $\supfg{f}{g}$ is lower semicontinuous.

{\itshape 2}. As $g$ is level-bounded and lower semicontinuous, the optimal value $\inf_{\rd} g$ is finite and for some $x\in\rd$, we have $g(x)=\inf_{\rd} g$. Hence, for all $s\geq \inf_{\rd} g$, the set $[g\leq s]$ is nonempty and compact (by assumption). Then, as $f$ is an upper semicontinuous function, for all $s\geq \inf_{\rd} g$, the supremum of $f$ on $[g\leq s]$ is finite and there exists $x\in\rd$ such that $g(x)\leq s$ and $f(x)=\sup_{[g\leq s]} f$. 
 
Let us prove that $\supfg{f}{g}$ is upper semicontinuous. 
Now, to use Berge's theorem, we must prove that the multivalued map $\rr\ni s\mapsto M(s):=\{x\in\rd : g(x)\leq s\}$ is upper hemicontinuous. Let $F$ be a closed subset of $\rd$, let $s\in\rr$ and $(s_n)_{n\in\nn}$ such that $(s_n)_{n\in\nn}$ converges to $s$ and $M(s_n)\cap F\neq \emptyset$ for all $n\in\nn$. Let us prove that $M(s)\cap F\neq \emptyset$. 
For  all $n\in\nn$, we have $M(s_n)\cap F \subseteq M(s^*)\cap F$ with $s^*=\sup_{n\in\nn} s_n$. There exists $y_n\in  M(s_n)\cap F$ which converges up to a subsequence to $y^*\in M(s^*)\cap F$ which is compact. As $g$ is lower semicontinuous $g(y^*)\leq \liminf_{k\to +\infty} g(y_{n_k})\leq \liminf_{k\to +\infty}s_{n_k}=s$ and $y^*\in F$. This proves that $M(s)\cap F\neq \emptyset$ and $M$ is upper hemicontinuous and by Berge's theorem $\supfg{f}{g}$ is upper semicontinuous. 

Finally, suppose that $f$ is, moreover, level-bounded. We have to prove that $\lim_{s\to +\infty} \supfg{f}{g}(s)=+\infty$. This is the same as for all $t\in\rr$ there exists $s^*\in\rr$ such that for all $s\in \rr$ satisfying $\supfg{f}{g}(s)\leq t$, we have $s\leq s^*$. So let $t\in \rr$ and let $s$ such that $\supfg{f}{g}(s)\leq t$. This implies that for all $x\in\rd$, such that $g(x)\leq s$, we have $f(x)\leq t$. We define $K:=\sup\{\norm{x} : f(x)\leq t\}$ which exists as $f$ is supposed to be level-bounded. Now we define $s^*:=\inf\{g(x) : \norm{x}=K+1\}$. We must have $s\leq s^*$. Indeed, otherwise, there exists $x^*$ such that $g(x^*)=s^*<s$ and $\norm{x^*}=K+1$. It follows that, for this $x^*$, we have $f(x^*)\leq t$ and thus $\norm{x^*}\leq K$ which is not possible.  
\end{proof}

We highlight the fact that if $f$ is lower semicontinuous then $\supfg{f}{g}$ can achieve $+\infty$ and being lower semicontinuous. For example, let us define $f(x)=1/x$ for all $x>0$ and 0 otherwise. Let $g$ be an increasing, lower semicontinuous and nonnegative function. Then $\supfg{f}{g}$ is finite on $\rr_-$ and equals to $+\infty$ on $\rr_+^*$. 

From the relation depicted at Eq.~\eqref{infsuprel} and the fact that $h$ is upper semicontinuous if and only if $-h$ is lower semicontinuous, we extend the results of Th.~\ref{continuity} to $\inffg{f}{g}$.     
\begin{corollary}[Continuity results for $\inffg{f}{g}$ - 1]
\label{infcontinuity}
Let $f:\rd\mapsto \rr$ and $g:\rd\mapsto \rr$. 
\begin{enumerate}
\item If $f$ and $g$ are upper semicontinuous, $\inffg{f}{g}$ is upper semicontinuous;
\item If $g$ is upper semicontinuous, $-g$ is level-bounded and $f$ is lower semicontinuous, then $\inffg{f}{g}$ is lower semicontinuous, finite-valued and $\inf$ in the definition of $\inffg{f}{g}$ can be replaced by a $\min$. 
\end{enumerate}
Finally, if $f$ is continuous, $g$ is upper semicontinuous and $-g$ level-bounded then $\inffg{f}{g}$ is continuous.
\end{corollary}    
The only result that we cannot extend is the asymptotic behaviour of $\inffg{f}{g}$ at $+\infty$. The direct application of Th.~\ref{continuity} is the fact that $\lim_{s\to -\infty} \inffg{f}{g}(s)=-\infty$. The second statement of Corollary~\ref{infcontinuity} is not useful to prove that the function $\underline{\alpha}^f$ defined at Eq.~\eqref{hahndef} is continuous when $f$ and $g$ are continuous, $-\norm{\cdot}$ not being level-bounded. 

Therefore, we develop a direct technique to prove that $\inffg{f}{g}$ is lower semicontinuous when $f$ is lower semicontinuous and $g$ is upper semicontinuous. 

\begin{theorem}[Extended lower semicontinuity for $\inffg{f}{g}$]
Let $g:\rd\mapsto \rr$ be an upper semicontinuous function. If the following assertions hold:
\begin{enumerate}
\item If $\sup_{\rd} g<+\infty$, there exists $x\in\rd$ such that $g(x)=\sup_{\rd} g$;
\item There exists a set-valued map $C:\rr\rightrightarrows \rd$ compact-valued and such that for all $s>\inf_{\rd} g$ and all $t\in \rr\cap [s,\sup_{\rd} g]$:
\begin{equation}
\label{compactnesscond}
\inffg{f}{g}(s)=\inf_{x\in [s\leq g]\cap C(t)} f(x)
\end{equation}
\item $f$ is lower semicontinuous.
\end{enumerate}
Then $\inffg{f}{g}$ is lower semicontinuous.
\end{theorem}

\begin{proof}
Let us take $s\in\rr$.

Suppose that $\inf_{\rd} g>-\infty$. If $s\leq \inf_{\rd} g$, $\inffg{f}{g}(s)$ is equal to $\inf_{\rd} f$. For all $t\in\rr$, we have $\inffg{f}{g}(t)\geq \inf_{\rd} f$. Hence any sequence $(s_n)_{n\in\nn}$ that converges to $s$ satisfies $\liminf_{n\to +\infty} \inffg{f}{g}(s_n)\geq \inffg{f}{g}(s)$ and we conclude that $\inffg{f}{g}$ is lower semicontinuous at $s$.

Assume that $\sup_{\rd} g$ is finite. If $s>\sup_{\rd} g$, the set $[s\leq g]$ is empty and thus the function $\inffg{f}{g}$ is constant and equal to $+\infty$ on $(\sup_{\rd} g,+\infty)$. Any sequence $(s_n)_{n\in\nn}$ converging to $s$ is eventually in $(\sup_{\rd} g,+\infty)$ and thus the limit of $(\inffg{f}{g}(s_n))_{n\in\nn}$ is equal to $+\infty$. 

Now, assume $s\in (\inf_{\rd} g,\sup_{\rd} g]$ if $\sup_{\rd} g$ is finite or $s\in  (\inf_{\rd} g,+\infty)$ if $\sup_{\rd} g=+\infty$ and consider a sequence $(s_n)_{n\in\nn}$ converging to $s$. We can assume that all the terms $s_n$ are strictly greater than $\inf_{\rd} g$. Let us pick a subsequence $(s_{\varphi_1(n)})_{n\in\nn}$ for which $\lim_{n\to +\infty} \inffg{f}{g}(s_{\varphi_1(n)})=\liminf_{n\to +\infty} \inffg{f}{g}(s_{n})$. We have either $\liminf_{n\to +\infty} \inffg{f}{g}(s_n)=+\infty$ (if $s=\sup_{\rd} g$ and $s_n$ is above $s$) or $\liminf_{n\to +\infty} \inffg{f}{g}(s_n)<+\infty$. For the first case, the lower semicontinuity at $s$ is obvious. Let us consider the second case, this is equivalent to $\lim_{n\to +\infty} \inffg{f}{g}(s_{\varphi_1(n)})=\liminf_{n\to +\infty} \inffg{f}{g}(s_{\varphi_1(n)})<+\infty$. Hence, this implies that for all $n\in\nn$, there exists $k\geq n$, $\inffg{f}{g}(s_{\varphi_1(k)})<+\infty$ which is equivalent to for all $n\in\nn$, there exists $k\geq n$, $s_{\varphi_1(k)}\leq \min\{\sup_{n\in\nn} s_n,\sup_{\rd} g\}$. Then there exists a subsequence $(s_{\varphi_1\circ \varphi_2(k)})_{k\in\nn}$ of $(s_{\varphi_1(k)})_{k\in\nn}$ for which
$s_{\varphi_1\circ \varphi_2(k)}\leq \min\{\sup_{n\in\nn} s_n,\sup_{\rd} g\}$ for all $k\in\nn$. Note that to take $\min\{\sup_{n\in\nn} s_n,\sup_{\rd} g\}$ serves to bound the sequence when $\sup_{\rd} g=+\infty$. In this case, the extraction is not necessary. When $\sup_{\rd} g$ is finite, some terms $s_n$ may be strictly greater than $\sup_{\rd} g$, but we want to keep those smaller than $\sup_{\rd} g$ and the extraction is mandatory in this case. Now, we define $s^*:=\sup_{k\in\nn}  s_{\varphi_1\circ \varphi_2(k)}\leq \sup_{\rd} g$. To simplify, let us rename temporarily this subsequence $(s_k)_{k\in\nn}$. 

By assumption, for all $k\in\nn$, $\inffg{f}{g}(s_k)=\inf\{f(x) : x\in [s_k\leq g]\cap C(s^*)\}$. As $g$ is upper semicontinuous, $s_k\leq \sup_{\rd}(g)$, $C(s^*)$ is compact and $[s_k\leq g]\cap C(s^*)\neq \emptyset$ then  $[s_k\leq g]\cap C(s^*)\neq \emptyset$ is nonempty and compact. As $f$ is lower semicontinuous, there exists $x_k\in [s_k\leq g]\cap C(s^*)$ such that $f(x_k)=\inffg{f}{g}(s_k)$. By compactness, there exists a subsequence $(x_{\psi(k)})_{k\in\nn}\in C(s^*)$ such that $x_{\psi(k)}\in [s_k\leq g]$ which converges to $x^*\in C(s^*)$. Moreover, as $g$ is upper semicontinuous, $s=\limsup_{k\to +\infty} s_{\psi(k)}\leq \limsup_{k\to +\infty} g(x_{\psi(k)})\leq g(x^*)$ and thus $x^*\in [s\leq g]\cap C(s^*)$. We conclude that $f(x^*)\geq \inffg{f}{g}(s)$. Finally, recalling that $s_k$ is actually $s_{\varphi_1\circ \varphi_2(k)}$ and as $f$ is lower semicontinuous, we have:
\[
\liminf_{n\to +\infty} \inffg{f}{g}(s_n)=\lim_{n\to +\infty} \inffg{f}{g}(s_{\varphi(n)})=\lim_{n\to +\infty} f(x_{\varphi(n)})\geq f(x^*)\geq\inffg{f}{g}(s)
\]
where $\varphi=\varphi_1\circ\varphi_2\circ \psi$.
\end{proof}

\begin{proposition}
Let assume that $g$ is upper semicontinuous unbounded from above or bounded from above with a global maximizer.
If $f$ is level-bounded and lower semicontinuous then Assumption~\eqref{compactnesscond} holds.
\end{proposition}

\begin{proof}
It suffices to define the set-valued map, for all $t\in\rr$:
\[
C(t):=\left\{
\begin{array}{lr}
\emptyset & \text{ if } \displaystyle{t>\sup_{\rd} g}\\
\displaystyle{\{y\in\rd : f(y)\leq f(x_t)\}} \text{ where } \displaystyle{t\leq g(x_t)} & \text{ if } \displaystyle{t\in\rr \cap [-\infty,\sup_{\rd} g]}
\end{array}
\right.
\] 
\end{proof}

\begin{corollary}
The function $\underline{\alpha}^f$ and $\overline{\alpha}^f$ of Eq.~\eqref{hahndef}
are increasing, continuous and level-bounded when $f$ is continuous and level-bounded.
\end{corollary}

%\ForAuthors{
%Pour l'exemple, dire que c'est difficilement vérifiable en général. Il existe une suite minimisante pour laquelle xn tend vers 0 et f(xn)=f inf g(sn) pour sn strictement positif qui tend vers 0. donc liminf f(xn)=liminf f inf g (sn)=inf f = f inf g (0). Difficilement généralisable.
%}
%
%\ForAuthors{
%Il manque un th avec $f$ upper semicontinuous and $g$ continuous. $\inffg{f}{g}$ can be lower semicontinuous $f(x)\leq t$ for all x st $g(x)\geq s$ s'il est fermé
%}

\section{Conclusion and Perspectives}
\label{conclusion}
In this paper, we have formally proved continuity properties for comparison type functions. This extends the results obtained for classical functions used in stability theory. We also study two algebraic properties : the level-boundedness and the positive definiteness of those parametric optimization value functions.

Future works should contain a study of continuity results of $\inffg{f}{g}$ for the case where $f$ is upper semicontinuous and $g$ is continuous and level-bounded. Indeed, in Example~\ref{exmupper}, the function $f$ is upper semicontinuous at 0 (continuous outside 0) and $g$ is continuous and level-bounded but the function $\inffg{f}{g}$ is continuous everywhere. The future works should explain if this is due to the example particularity or this result could be generalized.

\vspace{0,2cm}

\noindent {\bf Data Availability Statement}. Data availability is not applicable to this article as no new data were created or analysed in this study.

\bibliographystyle{alpha} 
\bibliography{khalilshortbib}

\end{document}